\documentclass[12pt]{amsart}
\usepackage{color}
\newif\ifprivate

\def\???{\ifprivate {\bf {???}} \marginpar{{\Huge {\bf ?}}}
\else \fi}

\def\???{{\bf {???}} \marginpar{{\Huge {\bf ?}}} }
\newcount\minleft
\newcount\timehour
\def\thetime{\timehour=\time
\divide\timehour by60 \minleft=\timehour \multiply\minleft by -60
\advance\minleft by\time \ifnum\time>720\advance\timehour
by-12\fi\relax
\number\timehour:\ifnum\minleft<10 %
    0\fi\relax\number\minleft
    \ifnum\time>720~pm \else~am\fi}

\newtheorem{defn0}{Definition}[section]
\newtheorem{prop0}[defn0]{Proposition}
\newtheorem{thm0}[defn0]{Theorem}
\newtheorem{lemma0}[defn0]{Lemma}
\newtheorem{corollary0}[defn0]{Corollary}
\newtheorem{example0}[defn0]{Example}
\newtheorem{remark0}[defn0]{Remark}
\newtheorem{conjecture0}[defn0]{Conjecture}

\newenvironment{definition}{ \begin{defn0}}{\end{defn0}}
\newenvironment{proposition}{\bigskip \begin{prop0}}{\end{prop0}}
\newenvironment{theorem}{\bigskip \begin{thm0}}{\end{thm0}}
\newenvironment{lemma}{\bigskip \begin{lemma0}}{\end{lemma0}}
\newenvironment{corollary}{\bigskip \begin{corollary0}}{\end{corollary0}}

\newenvironment{remark}{ \begin{remark0}\rm}{\end{remark0}}

\newcommand{\propref}[1]{Proposition~\ref{#1}}
\newcommand{\thmref}[1]{Theorem~\ref{#1}}
\newcommand{\lemref}[1]{Lemma~\ref{#1}}
\newcommand{\corref}[1]{Corollary~\ref{#1}}

\def\maxi{{\mathfrak m}}                   
\def\res{{\bf k}}                   
\def\HiRIM#1{H^i_{\mathcal M}(#1)}

\def\LC#1#2#3{H^{#1}_{#2}(#3)}

\def\Proj{\;{{\bf Proj}\,}}
\def\Spec{\;{{\bf Spec}\,}}

\def\av{\underline{a}}
\def\bv{\underline{b}}
\def\cv{\underline{c}}
\def\dv{\underline{d}}
\def\nv{\underline{n}}
\def\ov{\underline{0}}
\def\rv{\underline{r}}

\def\tv{\underline{t}}

\def\kv{\underline{k}}

\def\betav{\underline{\beta}}
\def\etav{\underline{\eta}}
\def\alphav{\underline{\alpha}}
\def\deltav{\underline{\delta}}
\def\varepsilonv{\underline{\varepsilon}}
\def\Zr{\mathbb Z^r} \def\Nr{\mathbb N^r}

\def\MM{\mathcal{M}}
\newcommand{\depth}{\operatorname{depth}}
\newcommand{\gdepth}{\operatorname{gdepth}}
\newcommand{\fg}{\operatorname{fg}}
\newcommand{\Gfg}{\operatorname{\Gamma-fg}}
\newcommand{\pcmd}{\operatorname{pcmd}}
\newcommand{\h}{\operatorname{ht}}

\begin{document}

\title[Non-standard multigraded modules]{{\bf
Cohomological properties of non-standard \\ \quad \\ multigraded modules}}
\author[Gemma Colom\'e-Nin]{Gemma Colom\'e-Nin}
\author[Juan Elias]{Juan Elias}
\thanks{Partially supported by MEC-FEDER MTM2007-67493. }
\address{Departament d'\`{A}lgebra i Geometria
\newline \indent Facultat de Matem\`{a}tiques
\newline \indent Universitat de Barcelona
\newline \indent Gran Via 585, 08007
Barcelona, Spain}
\email{\tt gcolome@ub.edu, elias@ub.edu}
\date{December 13, 2007}

\begin{abstract}
In this paper we study some cohomological properties of non-standard
multigraded modules and Veronese transforms of them. Among others
numerical characters, we study the generalized depth of a module and
we see that it is invariant by taking a Veronese transform. We prove
some vanishing theorems for the local cohomology modules of a
multigraded module; as a corollary  of these results we get that the
depth of a Veronese module is asymptotically constant.
\end{abstract}

\maketitle

\baselineskip 16pt


\section*{Introduction}

In commutative algebra, graded modules are object of study for many authors as
well as standard multigraded ones. For graded modules it has been studied
also the non-standard case, however then non-standard multigraded study is not so common. A
general reference on the subject could be \cite{GWII78}.

Along this paper $S$ is a non-standard $\mathbb N^r$-graded
$S_{\ov}$-algebra finitely generated by elements of multidegrees
$\gamma_i=(\gamma_1^i,\dots,\gamma_i^i,0\dots,0)\in\Nr$, with $\gamma_i^i\neq
0$, for $i=1,\dots,r$. For some of the results in the second part of the paper,
we need to restrict our setting to the almost-standard case, which is with positive
multiples of the canonical basis of $\mathbb R^r$ as a multidegrees of the
generators.

The main purpose of this paper is to study some cohomological properties of
multigraded $S$-modules and, in particular, of the Veronese modules associated to
a non-standard multigraded $S$-module $M$. We mainly study the vanishing of the
local cohomology modules of $M$ and of Veronese modules of $M$, generalizing
some results on the depth of Veronese modules associated to Rees algebras
proved in \cite{Eli04}.

In the section 1 we extend several results on homogeneous ideals of $\mathbb Z$-graded rings to
homogeneous ideals of non-standard $\mathbb Z^r$-graded rings, \propref{nongr}.
 By considering the multigraded
scheme ${\bf Proj}^r(S)$ we define the projective Cohen-Macaulay deviation of a
multigraded modules and we link this number with the generalized depth, studied
by Brodmann and Faltings (see \cite{Bro83} and \cite{Fal78}),
\thmref{threetenors1}. As a corollary we prove that the generalized depth
remains invariant by taking Veronese modules, \propref{fgver}.

In the first part of section two we prove, under the general  hypothesis on the
degrees of $S$, that the depth of the Veronese modules $M^{(\bv)}$ is constant
for special asymptotic values of $\bv$, \propref{gg}. In the second part of the
section we extend to a non-standard framework the notion of finitely
generation, \cite{Mar95}. Under some special degrees of $S$ we prove that the
generalized depth of a multigraded module coincides with its  finitely
graduation order, \thmref{threetenors}. We use it to get that the depth of the
Veronese  modules $M^{(\av,\bv)}$ is constant for large $\av,\bv\in \mathbb
N^r$, \thmref{main}, and we apply this result to the multigraded Rees algebras
associated to a finite set of ideals, \propref{depthmultirees}.

See \cite{hyr99} and its reference list for more results on the
Cohen-Macaulay and Gorenstein property of the multigraded Rees
algebras.

\bigskip
\noindent
{\sc Notations.}
Along the paper we use the underline to denote a multi-index: $\av
=(a_1,\cdots,a_r)\in \mathbb Z^r$. We write $|\av|=\sum_{i=1}^r  |
a_i|$. Given $\av, \bv \in \mathbb Z^r$, $\av  . \bv$ is the
termwise product of $\av$ and $\bv$, and $\av\ge\bv$ if, and only
if, $a_i\ge b_i$ for all $i=1,\dots,r$.
For all $\lambda\in \mathbb Z$ we put $\underline\lambda =(\lambda,\cdots, \lambda)\in \mathbb Z^r$.

Given integral vectors $\gamma_i=(\gamma_1^i,\dots,\gamma_i^i,0,\dots, 0)\in
\mathbb N^r$, $i=1,\cdots , r$,  such that $\gamma_i^i\neq 0 $,
we denote by
$\phi$ the map
$$
\begin{array}{cccc}
  \phi: & \mathbb Z^r & \longrightarrow & \mathbb Z^r \\
        & \nv         & \longmapsto     & \sum_{i=1}^r n_i \gamma_i
\end{array}
$$
notice that $Im(\phi)=\Gamma(\gamma_1,\cdots, \gamma_r)$ is  the
subgroup of $\mathbb Z^r$ generated by $\gamma_i$, $i=1,\cdots ,
r$.

We will denote by
 $G$ the $r\times r$ triangular  matrix whose columns are the vectors
$\gamma_1,\dots,\gamma_r$. Notice that $G$  is a non-singular matrix and that
the multi-index $t_1\gamma_1+\cdots + t_r \gamma_r$ is the  column vector
$G \tv$.

Given $\av\in \mathbb N^{*r}$ we denote by $\phi_{\av}$ the map
$$
\begin{array}{cccc}
  \phi_{\av}: & \mathbb Z^r & \longrightarrow & \mathbb Z^r \\
        & \nv         & \longmapsto     & \phi_{\av}(\nv)=\phi(\nv . \av ),
\end{array}
$$
with  $\phi_{\av}(\nv)=\phi(\nv . \av )=\sum_{i=1}^r (n_i a_i)
\gamma_i$ for all $\nv\in \mathbb Z^r$.

\bigskip

Let $S=\bigoplus_{\nv \in \mathbb N^r}S_{\nv} $ be a Noetherian
$\mathbb N^r$-graded ring  generated as $S_{\ov}$-algebra by
homogeneous elements $g_i^j$, $j=1,\cdots,\mu_i$, of multidegree
$\gamma_i$ for $i=1,\cdots, r$; the number of generators of $S$ is
$\mu=\mu_1+\cdots+\mu_r$. Notice that $S=\oplus_{\nv \in \Gamma}
S_{\nv}$, with $\Gamma=\Gamma(\gamma_1,\cdots, \gamma_r)$. We assume
that $S_{\underline{0}}$ is a local ring with maximal ideal $\maxi$
and infinite residue field.
\bigskip

For $i=1,\dots,r$, let $I_i$ be the ideal of $S$ generated by the
homogeneous components of $S$ of multidegrees
$(d_1,\dots,d_i,0,\dots,0)$ with $d_i\ne 0$. We define the
irrelevant ideal of $S$ as $S_{++}=I_1 \cdots I_r$. As usual we
write $S_+=\oplus_{\nv\neq\ov}S_{\nv} \supset S_{++}$.
Notice that in the graded case, i.e. $r=1$, these two ideals are the same $S_{+}=S_{++}$.

\bigskip
The Veronese transform of $S$ with respect to $\av\in\mathbb{N}^{*r}$, or
$(\av)$-Veronese, is the ring
$$
S^{(\av)}=\bigoplus_{\nv \in \mathbb{N}^r} S_{\phi_{\av}(\nv)}.
$$
This is a subring of $S$. The degrees of the generators  of
$S^{(\av)}$ have the same triangular configuration as the degrees of
$S$.

Given an $S$-graded module $M$ we denote by $M^{(\av,\bv)}$ the Veronese
transform of $M$ with respect to $\av,\bv\in\mathbb{N}^{*r}$, or
$(\av,\bv)$-Veronese,
$$
M^{(\av,\bv)}=\bigoplus_{\nv \in \mathbb Z^r}
M_{\phi_{\av}(\nv)+\bv}.
$$
This is an $S^{(\av)}$-module. Notice that in the case of
$\bv=(0,\dots,0)$ we get the classical definition of Veronese of a
module.

\bigskip

Let $M$ be a finitely generated $S$-module. By using a similar argument as in
\cite{HHR93}, Lemma 1.13 and Lemma 1.14, see also \cite{GWII78}, we can prove
that the
local cohomology functor and the Veronese functor commute
$$
\LC{*}{\mathcal M^{(\av)}}{M^{(\av,\bv)}} \cong (\LC{*}{\mathcal
M}{M})^{(\av,\bv)}
$$
where $\mathcal M$ is the maximal homogeneous  ideal of $S$, i.e.
$\MM=\mathfrak{m}\oplus S_{+}$, and $\av,\bv \in \mathbb N^{*r}$.
For the basic properties of local cohomology we use \cite{BSLC}
 as general reference.


\bigskip
\section{Generalized depth and Veronese modules.}

In this section, we study, in our multigraded setting, some properties of a
multigraded module and the Veronese transform of a module. These properties
allow to us to study the generalized depth of a multigraded module and its
Veronese.
\bigskip

Let $\Proj^r(S)$ be the set of all relevant homogeneous prime ideals on $S$,
which is the set of all homogeneous prime ideals $p$ of $S$ such that $p
\not\supset S_{++}$. Notice that $p \not\supset S_{++}$ if and only if for each
$1\le i\le r$ there exists  $1\le j(i)\le \mu_i$ such that $g_i^{j(i)}\notin
p$. Given an homogeneous ideal $p\subset S$ we denote by $U$ the multiplicative
closed subset of $S$ formed by the homogeneous elements of $S\setminus p$;
 we denote by  $S_{(p)}$ the set of fractions $m/s\in U^{-1} S$
such that $\deg(m)=\deg(s)\in \mathbb N^r$;
$S_{(p)}$ is a local ring with maximal ideal $p \; U^{-1} S \cap  S_{(p)}$.

\medskip
In the next proposition we prove several
results relating properties of non-standard
$\mathbb Z^r$-graded rings and modules with their Veronese transforms.

\medskip
\begin{proposition}
\label{nongr} $(i)$
 For all $p\in \Proj^r (S)$ the ring extension
$$
S_{(p)} \longrightarrow S_{p}
$$
is faithfully flat with closed  fiber $\res(p)$.

\noindent $(ii)$ For all $\av \in \mathbb N^{*r}$, the extension
$S^{(\av)} \hookrightarrow S$ is integral, $\dim(S^{(\av)})=\dim(S)$
and there is an homeomorphism of topological spaces
$$
\Proj^r(S^{(\av)})\cong \Proj^r(S).
$$
For all $p\in \Proj^r (S)$ it holds $\h(p^{(\av)})=\h(p)$.

\medskip
\noindent
$(iii)$ Let $M$ be a finitely generated $S$-module.
For all $p\in \Proj^r (S)$
and $\bv\in\Nr$, it holds $M^{(\av,\bv)}_{(p^{(\av)})}=M(\bv)_{(p)}.$
\end{proposition}
\begin{proof}
$(i)$
Since $p \not\supset S_{++}$,
 for each $i\in \{1,\cdots , r\}$ there exists a generator
$g_i^{j(i)}\notin p$, $1\le j(i) \le \mu_i$.
In particular  $g_i^{j(i)}\in U$ for all
$i=1,\cdots,r$.

 Let us consider the ring map
$$
\varphi: S_{(p)}[T_1,T_1^{-1},\cdots,T_r,T_r^{-1}]\longrightarrow U^{-1} S
$$

\noindent
defined by $\varphi(T_i)=g_i^{j(i)}$ and $\varphi(T_i^{-1})=(g_i^{j(i)})^{-1}$,
$i=1,\cdots, r$. We will prove that $\varphi$ is a ring isomorphism.

Let $m/s$ be a fraction of $U^{-1} S$; let $D=\sum_{i=1}^r D_i \gamma_i$, $D_i\in \mathbb N$,
 be the degree of $m$, and let $d=\sum_{i=1}^r d_i \gamma_i$, $d_i\in \mathbb N$, be the degree of $s$.
We define
$$
t=\prod_{i=1}^r (g_i^{j(i)})^{d_i-D_i}
$$
Hence, let us consider the identity
$$
\frac{m}{s}=\left(\frac{m}{s} t \right) t^{-1}.
$$
Notice that $\frac{m}{s} t\in S_{(p)}$ and that $t^{-1}=\varphi(\prod_{i=1}^r
T_i^{D_i-d_i})$, so $\varphi$ is an epimorphism.

Let $z=\sum_{\nv\in\mathbb Z^r} c_{\nv} T^{\nv}$ be an element of the ring
$S_{(p)}[T_1,T_1^{-1},\cdots,T_r,T_r^{-1}]$ such that $\varphi(z)=\sum_{\nv\in\mathbb Z^r}
c_{\nv}\prod_{i=1}^{r}(g_i^{j(i)})^{n_i}=0$, $\nv=(n_1,\dots,n_r)$.
Since $c_{\nv}\in S_{(p)}$, we can write $c_{\nv}=a_{\nv}/b_{\nv}$ with
$\deg(a_{\nv})=\deg(b_{\nv})$, $a_{\nv}\in S$ and $b_{\nv}\notin p$.
We write
 $$\prod_{i=1}^{r}
(g_i^{j(i)})^{n_i}=\frac{(g_{i_1}^{j(i_1)})^{n_{i_1}}\dots(g_{i_s}^{j(i_s)})^{n_{i_s}}}{(g_{j_1}^{j(j_1)})^{-n_{j_1}}
\dots(g_{j_t}^{j(j_t)})^{-n_{j_t}}}
$$
with  $(g_{i_1}^{j(i_1)})^{n_{i_1}}\dots(g_{i_s}^{j(i_s)})^{n_{i_s}}\in S$ and
$(g_{j_1}^{j(j_1)})^{-n_{j_1}}\dots(g_{j_t}^{j(j_t)})^{-n_{j_t}}\in S\setminus
p$,
i.e. $n_{i_1},\dots,n_{i_s}\ge 0$ and $n_{j_1},\dots,n_{j_t}<0$.

Now,
$$
\varphi(z)=\sum_{\nv\in\mathbb Z^r}\frac{a_{\nv}(g_{i_1}^{j(i_1)})^{n_{i_1}}
\dots(g_{i_s}^{j(i_s)})^{n_{i_s}}}{b_{\nv}(g_{j_1}^{j(j_1)})^{-n_{j_1}}\dots(g_{j_t}^{j(j_t)})^{-n_{j_t}}}=0
$$
and by reducing to a common denominator we get
$$
\varphi(z)=\sum_{\nv}\frac{d_{\nv}}{b}(g_{i_1}^{j(i_1)})^{n_{i_1}}\dots(g_{i_s}^{j(i_s)})^{n_{i_s}}=0
$$
Now,
$\deg(b)=\deg(d_{\nv})+\sum_{k=1}^{t}-n_{j_k}\gamma_{j_k}$.

Hence there exist $\delta \in U$ such that,
$$
\sum_{\nv\in\mathbb Z^r}\delta
d_{\nv}(g_{i_1}^{j(i_1)})^{n_{i_1}}\dots(g_{i_s}^{j(i_s)})^{n_{i_s}}=0.
$$
We have that if
$A_{\nv}=\delta
d_{\nv}(g_{i_1}^{j(i_1)})^{n_{i_1}}\dots(g_{i_s}^{j(i_s)})^{n_{i_s}}\neq 0$,
then
$$
\deg(A_{\nv})=\deg(\delta)+\deg(d_{\nv})+\sum_{k=1}^{s}n_{i_k}\gamma_{i_k}=
\deg(\delta)+\deg(b)+\sum_{i=1}^r n_i\gamma_i.
$$

Since  the $r\times r$ matrix $(\gamma_1,\cdots,\gamma_r)$ is upper triangular
and non-singular, the  degrees $\deg(A_{\nv})$ are different when $\nv$ ranges
$\mathbb Z^r$.
Hence  we get $A_{\nv}=0$ for all $\nv\in \mathbb Z^r$.

Let us consider the following identities in $S_{(p)}$
$$
 c_{\nv}=\frac{a_{\nv}}{b_{\nv}}=\frac{d_{\nv}(g_{j_1}^{j(j_1)})^{-n_{j_1}}\dots(g_{j_t}^{j(j_t)})^{-n_{j_t}}}{b}=
\frac{A_{\nv}(g_{j_1}^{j(j_1)})^{-n_{j_1}}\dots(g_{j_t}^{j(j_t)})^{-n_{j_t}}}{b\delta(g_{i_1}^{j(i_1)})^{n_{i_1}}
\dots(g_{i_s}^{j(i_s)})^{n_{i_s}}}
=0,
$$
so $z=0$, and hence $\varphi$ is a monomorphism.

Let us consider the multiplicative closed subset
$W=U^{-1}S\setminus p U^{-1}S$ .
Then $S_p= W^{-1}[U^{-1}S]$, furthermore
$$
S_p= W^{-1}(S_{(p)}[T_1,T_1^{-1},\cdots,T_r,T_r^{-1}]).
$$
From this identity  we deduce that the ring extension
$S_{(p)} \longrightarrow S_{p}$
is faithfully flat.
A simply computation shows that the closed fiber of $S_{(p)} \longrightarrow S_{p}$
is $\res(p)$.

\medskip
\noindent
$(ii)$
First we prove that the ring extension
$$
S^{(\av)} \hookrightarrow S
$$
is integral. Let $x\in S$ be an element of degree $\nv\in \mathbb
Z^r$. We write $\nv=\sum_{i=1}^r b^i \gamma_i$, $a=a_1\cdots a_r$,
and $\rv=( \frac{a b^i}{a_i}; i=1,\cdots, r)$. Then it is easy to se
that $a \nv =\phi_{\av}(\rv)$, so $x^a\in
S_{\phi_{\av}(\rv)}=(S^{(\av)})_{\rv}$. Hence $x$ is a zero of
$f(T)=T^a- x^a\in S^{(\av)}[T]$. Therefore $S$ is integral over
$S^{(\av)}$ and then $\dim(S^{(\av)})=\dim(S)$.

Notice that  $p \not\supset S_{++}$ if and only if
$p^{(\av)}=p \cap S^{(\av)} \not\supset S^{(\av)}_{++}$,
so we can  define a continuous map
$$
\begin{array}{cccc}
  \psi :&  \Proj^r(S)& \longrightarrow & \Proj^r(S^{(\av)}) \\
        &     p &   \mapsto & p^{(\av)} \\
\end{array}
$$
this map is surjective and closed since the the extension
$S^{(\av)} \hookrightarrow S$
is integral.

The map $\psi$ is injective: let $p_1, p_2 \in \Proj^r(S)$ such that
$p_1^{(\av)}= p_2^{(\av)}$. Given $x\in p_1$, by the argument done
in $(ii)$ we have
$$
x^a \in p_1\cap S^{(\av)}=p_1^{(\av)}= p_2^{(\av)}\subset p_2,
$$
so $x\in p_2$, i.e. $p_1\subset p_2$. By the symmetry of the problem
we have $p_1=p_2$. Hence $\psi$ is an homeomorphism of topological
spaces.

 The identity $\h(p^{(\av)})=\h(p)$ follows from the
above homeomorphism.

\medskip
\noindent $(iii)$
Notice that we always have
$$
M^{(\av,\bv)}_{(p^{(\av)})}\subset M(\bv)_{(p)}.
$$
In fact, let  $m/s$ be an element of $ M^{(\av,\bv)}_{(p^{(\av)})}$,
it means that $m\in (M^{(\av,\bv)})_{\nv}=M_{\phi_{\av}(\nv)+\bv}$
and $s\in (S^{(\av)})_{\nv}=S_{\phi_{\av}(\nv)}$ but $s\notin
p^{(\av)}$, $\nv\in\Zr$. Since  $s\in S^{(\av)}\setminus p^{(\av)}$
we have that $s\notin p$, so  $m/s\in M(\bv)_{(p)}$.

Let $m/s\in M(\bv)_{(p)}$ be a fraction such that $\deg(m)-\bv=\deg(s)=\nv\in
\mathbb N^r$ and $s\notin p$.
Since $s$ is an homogeneous element of degree
$\nv$, we can decompose $s$ in a sum of monomials on the generators $g_{i}^{j}$ of $S$:
$s=s_1+\dots+s_t$, with $\deg(s_i)=\nv$ for all
$i=1,\dots,t$.
Since $s\in S\setminus p$, there exist $k\in\{1,\dots t\}$ such that
$s_k\notin p$.
If we write
$$
s_k=\prod_{i=1}^r \prod_{j=1}^{\mu_i} (g_i^j)^{d_i^j},
$$
$d_i^j\in \mathbb N$, so
$$
\deg(s_k)=\nv=\sum_{i=1}^r\left( \sum_{j=1}^{\mu_i}d_i^j\right)\gamma_i.
$$
Since $s_k\notin p$, for each coefficient
$\sigma_i=\sum_{j=1}^{\mu_i}d_i^j\neq
0$ there exist a generator $g_i^{j(i)}\notin p$.
Let  $\sigma_{i_l}$,
$l\in\{1,\cdots, e\}$, be such a non-zero coefficients.
For each $l=1,\cdots, e$, let
$c_{i_l}\in \mathbb N\setminus\{0\}$ and $f_{i_l}\in \mathbb N\setminus\{0\}$
be non-negative integers such that
$$
\sigma_{i_l} + c_{i_l}= f_{i_l} a_{i_l}.
$$
We put $c_i=f_i=0$ for all $i\notin \{ i_1,\cdots, i_e\}$. We define
$$
z=\prod_{l=1}^{e} (g_{i_l}^{j(i_l)})^{c_{i_l}}\notin p.
$$
Since $s\notin p$ is homogeneous, $zs\notin p$ is still homogeneous and then
$$
\deg(zm)-\bv=\deg(zs)=\deg(zs_k)=\sum_{i=1}^r f_{i}
a_{i}\gamma_i=\phi_{\av}((f_1,\dots,f_r)),
$$
so $m/s=(zm)/(zs)\in
M^{(\av,\bv)}_{(p^{(\av)})}$.

\end{proof}

Given an ideal $p\in \Spec(S)$ we denote by $p^*$ the  prime ideal generated by
the homogeneous elements belonging to $p$, see \cite{GWII78} section 2. We can
relate the depths of the localization on a prime $p$ with the localization on
$p^*$.

\begin{proposition}\label{depth+ht}
Let us assume that $S$ is a catenary ring.
Let $M$ be a finitely generated $\Zr$-graded $S$-module.
Given an ideal  $p\in \Spec(S)$  such that $p \not\supset S_{++}$ and $M_p\neq 0$, then
it holds
$$
\depth(M_p)+ \dim(S/p)=\depth(M_{(p^*)})+ \dim(S/p^*)
$$

\end{proposition}
\begin{proof}
We put $d=\dim(S_p/p^* S_p)$. From \cite{GWII78}, Proposition 1.2.2 and
Corollary 1.2.4, we have that $\depth(M_p)=\depth(M_{p^*})+ d$ and
$\dim(M_p)=\dim(M_{p^*})+ d$. On the other hand, since $S$ is catenary we have
$\dim(S_{p})=\dim(S)-\dim(S/p)$ and $\dim(S_{p^*})=\dim(S)-\dim(S/p^{*})$. From
these identities we get
$$
\begin{array}{lll}
\depth(M_p)+ \dim(S/p)& = &\depth(M_{p^*})+ d+ \dim(S)-\dim(S_p) \\ \\
  & = &\depth(M_{p^*})+ d+ \dim(S)-\dim(S_{p^*})-d \\ \\
  & = &\depth(M_{p^*})+ \dim(S/p^*).
\end{array}
$$

Since the morphism $S_{(p)}\longrightarrow S_p$ is faithfully flat with closed
fiber $\res(p)$ we get, \cite{MatCR} Theorem 23.3, that
$\depth(M_{p^*})=\depth(M_{(p^*)})$. From this we get the claim.
\end{proof}

\bigskip

Let $M$ be a $S$-graded module. We denote by $\pcmd(M)$ is the
\textit{projective Cohen-Macaulay deviation} of $M$, i.e.
the maximum of
$$
\dim(S_{(p)})- \depth(M_{(p)})
$$
where $p\in \Proj^r(S)$, see \cite{Eli04}.

We denote by $\gdepth(M)$ the so-called \textit{generalized depth} of $M$ with
respect to the homogeneous maximal  ideal $\mathcal M$ of $S$,
 $\gdepth(M)$ is the greatest integer $k \ge 0$ such that
$$
S_{++} \subset rad(Ann_{S}(\HiRIM{M}))
$$
 for all  $i < k$, see \cite{HM94b}.
 Notice that $\gdepth(M)\ge \depth(M)$.

\bigskip
In the case of being $S_{\ov}$ a quotient of a regular ring, we can relate
these last two integers. This relation is crucial in order to prove that the
generalized depth of a module coincides with the one of its Veronese transform.
Next theorem generalizes Proposition 2.2. in \cite{HM99}.

\begin{theorem}
\label{threetenors1} Let $M$ be a finitely generated $S$-graded
module.
If $S_{\ov}$ is the quotient of a regular ring then
$$
\gdepth(M)= \dim(S)- \pcmd(M).
$$
\end{theorem}
\begin{proof}
From \cite{Fal78}, Satz 1, (see also \cite{Mar95}) we get
$$
\gdepth(M)=\min_{p\in \Sigma}\{\depth(M_p)+ \dim(S/p)\}
$$
with $\Sigma=\{\mathfrak a\mid
\mathfrak{a}\in \Spec(S),\mathfrak a \not\supset
S_{++}\}$.
From
\propref{depth+ht}, we get that
$$
\depth(M_p)+\dim(S/p)=\depth(M_{(p^*)})+\dim(S/p^*),
$$
 so we can assume that
$p\in\Proj^r(S)$. Therefore we get
$$
\gdepth(M)=\min_{p\in\Proj^r(S)}\{\depth(M_{(p)})+\dim(S/p)\}.
$$
Since $S$ is catenary
 $\dim(S/p)=\dim(S)-\dim(S_{(p)})$, and hence
 $$
 \begin{array}{lll}
   \gdepth(M) & = & \dim(S)- \max_{p\in\Proj^r(S)}\{\dim(S_{(p)}) -\depth(M_{(p)})\}
   \\ \\
   & =& \dim(S)-\pcmd(M).
 \end{array}
$$
\end{proof}

Now, we can prove the invariance of gdepth under Veronese transforms:

\begin{corollary}
\label{fgver} Let us assume that $S_{\ov}$ is the quotient of a
regular ring and let $M$ be a finitely generated $S$-graded module,
then it holds
$$
\gdepth(M^{(\av,\bv)})=\gdepth(M)
$$
for all $\av,\bv\in\mathbb N^{*r}$.
\end{corollary}
\begin{proof}
From  \thmref{threetenors1} we get
$$
\gdepth(M^{(\av,\bv)})=\dim(S^{(\av)})-\pcmd(M^{(\av,\bv)}).
$$
and from \propref{nongr} (ii), $\dim(S^{(\av)})=\dim(S)$.
Now, again from \propref{nongr}(iii) we deduce
$$
\dim(S^{(\av)})-\pcmd(M^{(\av,\bv)})=\dim(S)-\pcmd(M(\bv)).
$$

Again from \thmref{threetenors1}
$\gdepth(M(\bv))=\dim(S)-\pcmd(M(\bv))$. Using the definition of
$\gdepth$ we have that $\gdepth(M(\bv))=\gdepth(M)$, and so we get
the claim
$$\gdepth(M^{(\av,\bv)})=\gdepth(M).$$
\end{proof}

\bigskip
\section{Vanishing theorems and Asymptotic depth of Veronese modules.}

In this section we introduce the generalization, in the multigraded case, of
the concept of $\fg(M)$, which in the graded case controls the finitely
graduation of the local cohomology modules of a graded module $M$ with respect
to the maximal homogeneous ideal of $S$. We prove some results on the vanishing of a
module and its local cohomology modules and we relate this with the generalized
depth. For that goal, we have to fit the generalization of $\fg$, that we call
$\Gfg$, to the multigraduation. We also study the asymptotic depth of Veronese
modules. We can prove that this depth is constant for $(\av,\bv)$-Veronese
modules for $\av,\bv$ in suitable asymptotic regions of $\Nr$ by using the
previous work done in the paper.

\bigskip
We want to study the depth of the Veronese modules $M^{(\av, \bv)}$ for large
values $\av, \bv\in\Nr$. Under the hypothesis on the multidegrees of this paper
we can prove the following results by considering some Veronese modules.

\medskip
We denote by $vad(M^{(*)})$ (resp. $vad(M^{(*,*)})$) the Veronese asymptotic
depth of $M$, that means the maximum of $\depth(M^{(\av)})$  (resp.
$\depth(M^{(\av,\bv)})$)  for all $\av\in\mathbb N^{*r}$ (resp. for all $\av,
\bv \in\mathbb N^{*r}$).

\begin{proposition}\label{gg}
Let $s=vad(M^{(*)})$. There exists $\av=(a_1,\dots,a_r)\in\mathbb N^{*r}$ such
that for all $\bv\in\{(\lambda_1 a_1,\dots,\lambda_r a_r)\mid
\lambda_i\in\mathbb N^*\}$
$$\depth(M^{(\bv)})=s$$
is constant.
\end{proposition}
\begin{proof}
Let $s=vad(M^{(*)})$, this means that there exist an
$\av\in\mathbb N^{*r}$ such that
$$\LC{i}{\mathcal{M}^{(\av)}}{M^{(\av)}}=0$$
for $i=0,\dots,s-1$.

Let us consider $\bv\in\{(\lambda_1 a_1,\dots,\lambda_r a_r)\mid
\lambda_i\in\mathbb N^*\}=\{\underline{\lambda}{\cdot}\av\mid
\underline{\lambda}\in\mathbb N^{*r}\}$. Then for all $\nv\in\Zr$,
since
$\phi_{\bv}(\nv)=\phi(\bv{\cdot}\nv)=\phi(\av{\cdot}\underline{\lambda}{\cdot}\nv)=\phi_{\av}(\underline{\lambda}{\cdot}\nv)$,
we have that
$$\LC{i}{\mathcal{M}^{(\bv)}}{M^{(\bv)}}_{\nv}=\LC{i}{\mathcal{M}}{M}_{\phi_{\bv}(\nv)}=\LC{i}{\mathcal{M}}{M}_{\phi_{\av}(\underline{\lambda}{\cdot}\nv)}=
\LC{i}{\mathcal{M}^{(\av)}}{M^{(\av)}}_{\underline{\lambda}{\cdot}\nv}=0$$
for $i=0,\dots,s-1$. From this, we deduce that
$\depth(M^{(\bv)})\ge s$, but $s$ was the maximum. Therefore,
$$\depth(M^{(\bv)})=s$$
for all $\bv\in\{(\lambda_1 a_1,\dots,\lambda_r a_r)\mid
\lambda_i\in\mathbb N^*\}$.
\end{proof}

\medskip
Let us consider the multigraded Rees algebra associated to ideals
$I_1,\dots,I_r$ in a Noetherian local ring $(R,\mathfrak{m})$,
$$
\mathcal{R}(I_1,\dots,I_r)=\bigoplus_{\nv\in\Nr}I_1^{n_1}T_1^{n_1}\cdots
I_r^{n_r}T_r^{n_r}\subset R[T_1,\dots,T_r].
$$

\begin{proposition}
Let $s=vad(\mathcal R(I_1,\dots,I_r)^{(*)})$.
 There exists $\av=(a_1,\dots,a_r)\in\mathbb N^{*r}$  such that for
all $\bv\in\{(\lambda_1 a_1,\dots,\lambda_r a_r)\mid
\lambda_i\in\mathbb N^*\}$
$$\depth({\mathcal R}(I_1^{b_1},\dots,I_r^{b_r}))=s.$$

\noindent Moreover, if $\depth(\mathcal R(I_1,\dots,I_r))=s$, then
$$
\depth(\mathcal R(I_1^{b_1},\dots,I_r^{b_r}))=s
$$
is constant for all $\bv\in\mathbb{N}^{*r}$.
\end{proposition}
\begin{proof}
 Notice that the multigraded Rees algebra has a standard
graduation and hence, for $\av=(a_1,\dots,a_r)$,
$$\mathcal{R}(I_1^{a_1},\dots,I_r^{a_r})=\mathcal{R}(I_1,\dots,I_r)^{(\av)}
$$
and then the claim is a consequence of the previous proposition.
The second statement follows from the first one by considering
$\av=(1,\dots,1)$.
\end{proof}

\bigskip
We would like to extend the previous results on the asymptotic depth of the
Veronese modules to regions of $\Nr$ instead of some nets there. First we have
to study the vanishing of the local cohomology modules of a multigraded module
$M$.

\bigskip
A cone $C_{\betav}\subset\mathbb N^r$ with vertex at $\betav \in \mathbb N^r$
with respect to $\gamma_1,\dots,\gamma_r$ is a region of $\mathbb N^r$ whose
points are of the form $\betav+\sum_{i=1}^{r}\lambda_i\gamma_i\in\mathbb{N}^r$
with $\lambda_i\in\mathbb{R}_{\ge 0}$ for $i=1,\dots,r$.

Given  $\nv=(n_1,\dots,n_r)\in\mathbb Z^r$ we denote
$\nv^{*}=(|n_1|,\dots,|n_r|)\in\mathbb{N}^r$.

If $M$ be a finitely generated $\Zr$-graded $S$-module  generators
$h_1,\dots,h_s$ of multidegrees $\dv^1=(d_1^1,\dots,d_r^1), \dots,
\dv^s=(d_1^s,\dots,d_r^s) \in \mathbb Z^r$ respectively then we denote by
$\Gamma_M$ the $\Gamma$-invariant subset of $\mathbb Z^r$
$$
\Gamma_M= \bigcup_{i=1}^{s} (\dv^i + \Gamma),
$$
i.e. $\mathbb Z ^r \setminus \Gamma_M$ is the set of multi-index for which
there is no non-zero elements of $M$.

\begin{lemma}
\label{asim} For all $\betav\in\mathbb Z^r$ and $c\in \mathbb N$ there exists
$\alphav \in \Gamma_M$ such that $\alphav\ge \cv=(c,\cdots,c)$ and $\alphav \in
\betav + \Gamma.$
\end{lemma}
\begin{proof}
The condition $\alphav \in (\betav + \Gamma)\cap (\dv^1 +\Gamma)$ can be
translated to the equation
$$
\alphav= \dv^1 + G \tv = \betav + G \nv,
$$
so
$$
\nv= \tv + G^{-1}(\dv^1-\betav).
$$
Hence for a $\tv \gg \underline{0}$ we get that $\nv\gg \underline{0}$, so
$\alphav \in \Gamma_M \cap (\betav +\Gamma)$ and $\alphav \ge \cv$.
\end{proof}

\begin{proposition}\label{S++^uM=0}
Let $M$ be a finitely generated $\Zr$-graded $S$-module such that
$S_{++}\subset rad(Ann_S(M))$. Then there exists
$\betav=(\beta_1,\dots,\beta_r)\in\Gamma_M$ such that $M_{\nv}=0$, for all
$\nv\in\Zr$ such that $\nv^{*}\in C_{\betav}$.
\end{proposition}
\begin{proof}
We prove the result first assuming that  $M$ is $\mathbb N^r$ generated, i.e.
we assume that $h_1,\dots,h_s$ are the generators of the $S$-module $M$ with
multidegrees $(d_1^1,\dots,d_r^1), \dots, (d_1^s,\dots,d_r^s) \in \mathbb N^r$
respectively. Let $\alphav=(\alpha_1,\dots,\alpha_r)\in\mathbb Z^r$ be the
maximum componentwise of these multidegrees, i.e. $\alpha_i=\max\{d_i^1,\cdots,
d_i^s\}$, $i=1,\cdots, r$.

The elements of $M_{\nv}$, $\nv\in\mathbb N^r$, are linear combinations with
coefficients on $S_{\ov}$
 of elements of the type
$$\underline{g_1}^{\underline{m_1}}\dots\underline{g_r}^{\underline{m_r}}h_j$$
where, using multiindex notation,
$\underline{g_t}^{\underline{m_t}}=({g_t^1})^{m_t^1}\cdots
({g_t^{\mu_t}})^{m_t^{\mu_t}}$ with
$\underline{m_t}=(m_t^1,\dots,m_t^{\mu_t})\in\mathbb N^{\mu_t}$. This element
has multidegree
$$
\nv=\deg(\underline{g_1}^{\underline{m_1}}\dots\underline{g_r}^{\underline{m_r}}h_j)
= G
\left(\begin{array}{c}|\underline{m_1}|\\\vdots\\|\underline{m_r}|\end{array}\right)+
\left(\begin{array}{c}d_r^j\\\vdots\\d_r^j\end{array}\right).
$$

Let $u$ be a non-negative integer such that $(S_{++})^{u}M=0$. We define
$\betav$ recursively:
$$
\beta_i=u \gamma_i^i + \beta_{i+1}  \gamma_{i}^{i+1} +\cdots + \beta_{r}
\gamma_{i}^{r}+ \alpha_i
$$
for $i=r,\cdots, 1$.

Given a multi-index $\nv=\beta+ \sum_{i=1}^r \lambda_i \gamma_i \in
C_{\betav}\cap  \Gamma_M$, $\lambda_i\ge 0$,
 we have to prove that $M_{\nv}=0$.
This is equivalent to prove that if
$$
\nv = G \left(\begin{array}{c}\lambda_1\\\vdots\\\lambda_1\end{array}\right)+
\left(\begin{array}{c}\beta_1\\\vdots\\\beta_r\end{array}\right) = G
\left(\begin{array}{c}|\underline{m_1}|\\\vdots\\|\underline{m_r}|\end{array}\right)+
\left(\begin{array}{c}d_1^j\\\vdots\\d_r^j\end{array}\right)
$$
then $|\underline{m_1}|\ge u+\lambda_1, \cdots , |\underline{m_r}|\ge u +
\lambda_r$.

We will prove  by recurrence a stronger result:
$$
\beta_i +\lambda_i \ge |\underline{m_i}|\ge u + \lambda_i
$$
for $i=1,\cdots ,r$. From the definition of  $ \beta_r=u \gamma_r^r + \alpha_r$
and
$$
\beta_r + \lambda_r \gamma_r^r =|\underline{m_r}| \gamma_r^r + d_r^j
$$
we deduce
$$
\gamma_r^r(|\underline{m_r}|-(u+\lambda_r))=\alpha_r- d_r^j\ge 0.
$$
Since $\gamma_r^r\ge 1$ we get
$$
|\underline{m_r}|\ge u+\lambda_r.
$$
On the other hand
$$
\beta_r + \lambda_r - |\underline{m_r}|= d_r^r +
(\gamma_r^r-1)(|\underline{m_r}|-\lambda_r)\ge 0
$$

Let us assume that $\beta_r +\lambda_r\ge |\underline{m_r}|\ge u+\lambda_r,
\cdots, \beta_{i+1} +\lambda_{i+1}\ge |\underline{m_{i+1}}|\ge u+\lambda_{i+1}$
we will prove that $\beta_{i}  +\lambda_{i}\ge |\underline{m_{i}}|\ge
u+\lambda_{i}$, $i\ge 1$. We have
$$
\beta_i+ \lambda_i \gamma_i^i + \lambda_{i+1}  \gamma_{i}^{i+1} +\cdots +
\lambda_{r}  \gamma_{i}^{r} = |\underline{m_{i}}| \gamma_i^i +
|\underline{m_{i+1}}|  \gamma_{i}^{i+1} +\cdots + |\underline{m_{r}}|
 \gamma_{i}^{r}+ d_i^j
$$
so
$$
 \gamma_i^i(u+\lambda_i-|\underline{m_{i}}|) +\sum_{l=i+1}^r \gamma_i^l (\beta_l+\lambda_l-|\underline{m_{l}}|)+
 \alpha_i - d_i^j= 0.
$$
By induction we deduce
$$
|\underline{m_{i}}|\ge u+\lambda_i.
$$
A simple computation shows that
$$
\beta_i + \lambda_i - |\underline{m_i}|=
(\gamma_i^i-1)(|\underline{m_i}|-\lambda_i)+ \sum_{l=i+1}^r \gamma_i^l
(|\underline{m_{l}}|-\lambda_l)+ d_i^j \ge 0.
$$
\noindent Hence we have proved that $M_{\nv}=0$ for all $\nv\in C_{\betav}$.

\medskip
Let us assume now that $M$ is generated by
 $h_1,\dots,h_s$ with
multidegrees $(d_1^1,\dots,d_r^1), \dots, (d_1^s,\dots,d_r^s) \in \mathbb Z^r$
respectively. Let $c= |\min\{ 0, d_i^j, j=1,\cdots, s, i=1,\cdots, r\}|$. Let
$N$ be the following submodule of $M$:
$$
N=\bigoplus_{\nv\ge \ov} M_{\nv}
$$
From \lemref{asim} there is $\alphav\in \Gamma_M$ such that $\alphav \ge \cv$
and $\alphav \in \Gamma_M \cap (\betav(N) +\Gamma).$ Since $C_{\alphav}\subset
C_{\betav}$  and $\alphav\ge \cv$ we get that  $M_{\nv}=0$ for all
$\nv\in\mathbb Z^r$ and $n^*\in C_{\betav}$.
\end{proof}

\begin{corollary}
\label{S++^u(M/N)=0} Let $M$ be a finitely generated $\Zr$-graded $S$-module
and $N\subset M$ a submodule. We assume that $(S_{++})^u(M/N)=0$ for
$u\in\mathbb Z$. Then there exists $\betav \in \Gamma_{M/N}$ such that
$M_{\nv}\subset N_{\nv}$, for all $\nv \in\Zr$ such that $\nv^{*} \in
C_{\betav}$.
\end{corollary}
\begin{proof}
It is only necessary to use \propref{S++^uM=0} with the finitely generated
module $M/N$. There will exists a cone $C_{\betav}$ where $(M/N)_{\nv}=0$ for
$\nv^{*}\in C_{\betav}$, and hence $M_{\nv}\subset N_{\nv}$.
\end{proof}

\bigskip
We say that a $S$-graded module $M$ is \textit{$\Gamma$-finitely
graded} if there exists a cone $C_{\betav}\subset\mathbb N^r$
where $M_{\nv}=0$ for all $\nv\in\mathbb Z^r$ such that
$\nv^{*}\in C_{\betav}$. We denote by $\Gfg(M)$ the greatest
integer $k \ge 0$ such that $ \LC{i}{\mathcal M}{M} $ if
$\Gamma$-finitely graded  for all  $i < k$, see \cite{Mar95}.

\begin{remark}
Notice that in the standard graded case, i.e. $r=1$,  the
definition of $\Gfg(M)$ coincides with the classical
$$\fg(M)=\max\{k\ge 0\mid \LC{i}{\mathcal M}{M} \text{ is finitely
graded for all }i<k\}.$$
 In this case  a module is
finitely graded if the pieces of degree $n$ are 0 for $|n|\ge
n_0$, for some $n_0\in\mathbb N$, which is, in fact, a cone with
vertex in $n_0$, so
$$\fg(M)=\Gfg(M).$$
\end{remark}

\bigskip
From now on we assume that the ordering is almost-standard. By
almost-standard multigraded (or $\Zr$-graded) ring $S$ we mean the
multigraded ring with generators of multidegrees
$$
\begin{array}{l}
\gamma_1=(\gamma_1^1,0,\dots,0)=\gamma_1^1e_1\\ ... \\
\gamma_i=(0,\dots,0,\gamma_i^i,0,\dots,0)=\gamma_i^i
e_i\\ ...\\ \gamma_r=(0,\dots,0,\gamma_r^r)=\gamma_r^re_r\\
\end{array}
$$
with $\gamma_1^1,\dots,\gamma_r^r > 0$ and $e_1,\dots,e_r$ the
canonical basis of $\mathbb{R}^r$.
Notice that in this case we have
$$
C_{\betav}= (\betav + (\mathbb R_{\ge 0})^r)\cap \mathbb N^r
$$
for all $\betav \in \mathbb Z^r$.
Notice that the intersection of two cones is a cone:
$$
C_{\alphav} \cap C_{\betav} = C_{\deltav}
$$
with $\delta=(\max\{\alpha_i, \beta_i\}; i=1,\cdots, r)$.

\bigskip
An important point in the proof of the main proposition, is to assure that
$H^k_{\MM}(M)$ is $\Gamma$-finitely graded for all $k\ge 0$ in case that the
module $M$ is $\Gamma$-finitely graded as well. For that reason we have to
restrict the graduation to the almost-standard case. We prove that in the next
proposition.

\begin{proposition}\label{local coho fgrad}
Let $M$ be a finitely generated $\Zr$-graded $S$-module. If $M$ is
$\Gamma$-finitely graded then $H^k_{\mathcal{M}}(M)$ is also
$\Gamma$-finitely graded for all $k\ge 0$.
\end{proposition}

\begin{proof}
Since $M$ is $\Gamma$-finitely graded, it means that there exist an element
$\betav\in\Nr$ such that $M_{\nv}=0$ for all $\nv\in\Zr$ with $\nv^*\in
C_{\betav}$. We want to prove that $H^k_{\mathcal{M}}(M)_{\nv}=0$ for
$\nv\in\Zr$ with $\nv^*\in C_{\betav}$ as well.

Since $H^0_{\mathcal{M}}(M)=\Gamma_{\mathcal{M}}(M)\subseteq M$, then we
have directly the claim for $k=0$.
Let us assume that $k>0$.

The ideal  $\mathcal{M}$ is generated by a system of generators of
$\maxi$, say $h_1,\cdots,h_v$, and by $g_i^j$, $j=1,\cdots, \mu_i$,
$i=1,\cdots,r$.
If we denote by $f_1,\cdots,f_{\sigma}$ the above system of generators of $\MM$ then   the local cohomology
modules $\LC{*}{\mathcal M}{M}$ are the cohomology modules of the
Koszul complex
$$
0\longrightarrow M
\longrightarrow \bigoplus_{i=1}^{\sigma} M_{f_i}
\longrightarrow \bigoplus_{1\le i < j\le \sigma} M_{f_i f_j}
\longrightarrow
\cdots
\longrightarrow
M_{f_1 \cdots  f_\sigma}
\longrightarrow
0.
$$
The module $\LC{k}{\mathcal M}{M}$ is  $S$-graded: the grading
is induced by the grading defined on the localizations
$M_g$, where  $g$ is an arbitrary  product of $k$ different generators of $\mathcal M$.
Given $z=x /g^t \in M_g$ we have
$$
\deg(z)=\deg\left(\frac{x}{g^t}\right)=\deg(x)- t\; \deg(g).
$$

If we assume that $\deg(z)=\nv$ with $\nv^*\in C_{\betav}$ then  there
exist a vector $\varepsilonv=(\varepsilon_1,\dots,\varepsilon_r)\in\{-1,+1\}^r$
such that
$\varepsilonv . \nv=\betav+G\underline{\lambda}$ with $\lambda_i\in\mathbb R_{\ge
0}$.
We denote here $\varepsilonv . \nv$ for the termwise product of $\varepsilonv$ and $\nv$.
So,
$$
\nv=\varepsilonv . (\betav +G\underline{\lambda}).
$$
On the other hand we may assume, without loss of generality, that
$\deg(g)=G \kv$ with $\kv=(k_1,\cdots, k_w,0,\cdots,0)$ with $k_i\neq 0$,
$i=1,\cdots,w$.
Hence we have
$$
\deg(x g^s)=
\deg(z)+(t+s) \deg(g)=
\varepsilonv . (\betav +G\underline{\lambda}) +
(t+s) G \kv
$$
for all $s\ge 0$.

We want to prove that $\deg(x g^s )^*\in C_{\betav}$, for some $s\ge
0$, so we have to assure that there exists
$\underline{\mu}\in(\mathbb R_{\ge 0})^r$  and $\etav\in \{-1,+1\}^r$ such that

$$
\etav . [\varepsilonv . (\betav +G\underline{\lambda}) +
(t+s) G \kv ]
=\betav+G\underline{\mu}.
$$

\noindent
For $i=w+1\cdots, r$ we have the equation
$$
\eta_i \; \varepsilon_i (\beta_i + \lambda_i \gamma_i^i)= \beta_i + \mu_i \gamma_i^i,
$$
we set $\eta_i=\varepsilon_i$ and $\mu_i=\lambda_i\ge 0$.

\noindent
For $i=1,\cdots, w$ we set $\eta_i=1$, and then we have to consider the equation
$$
\varepsilon_i (\beta_i + \lambda_i \gamma_i^i)+ (t+s) k_i \gamma_i^i= \beta_i + \mu_i \gamma_i^i.
$$
If $\varepsilon_i=1$ then
$$
\mu_i=\lambda_i + (t+s) k_i\ge 0.
$$
If $\varepsilon_i=-1$ then
$$
\mu_i= - 2 \; \frac{\beta_i}{\gamma_i^{i}} - \lambda_i + (t+s) k_i\ge 0
$$
for an integer $s\gg 0$.

We have proved that $H^k_{\mathcal{M}}(M)_{\nv}=0$ for $\nv\in\Zr$ with
$\nv^*\in C_{\betav}$, so $H^k_{\mathcal{M}}(M)$ is $\Gamma$-finitely graded.
\end{proof}

\bigskip
 In the next result we relate the
two integers attached to $M$ studied in the paper, $\gdepth(M)$ and $\Gfg(M)$.
The first part of the next result follows \cite{Mar95}, Proposition 2.3. or
\cite{TI89}, Lemma 2.2. Since these papers use extensively results on $\mathbb
Z$-graded modules we will adapt them in the almost-standard multigraded case
that we consider here.

\begin{theorem}
\label{threetenors}
Let $M$ be a finitely generated $S$-graded module, then it holds
$$
\Gfg(M)=\gdepth(M).
$$
\end{theorem}
\begin{proof}
First we prove the inequality $\Gfg(M)\le \gdepth(M)$.
If
$\LC{i}{\mathcal M}{M}$ is $\Gamma$-finitely graded then there
exists a cone $C_{\betav}$ with vertex in some $\betav\in\Nr$, such that $\LC{i}{\mathcal
M}{M}_{\nv}=0$ for all $\nv\in\Zr$ with $\nv^{*}\in C_{\betav}$.

We have to prove that $S_{++}\subset rad(Ann_S(\LC{i}{\mathcal M}{M}))$, i.e.
for all generator $x=g_1^{m_1}\cdots g_r^{m_r}$ of $S_{++}$,
$m_i\in\{1,\cdots, \mu_i\}, i=1,\cdots, r$,  we have to find a
suitable $a>0$ such that for all $\nv\in\Zr$, $x^a \LC{i}{\mathcal
M}{M}_{\nv}=0$.

If $\nv^*\in  C_{\betav}$ then $\LC{i}{\mathcal
M}{M}_{\nv}=0$, so  for all $a\ge 0$ it holds $x^a \LC{i}{\mathcal
M}{M}_{\nv}=0$.

\noindent
We put $a= 2 \max\{\beta_1,\cdots, \beta_r\}$.
Let us assume that $\nv^*\notin  C_{\betav}$.
That means that, without loss of generality, that
$-\beta_i < n_i < \beta_i$, $i=1,\dots, u$, and
$|n_i|\ge \beta_i$ for $i=u+1,\cdots, r$.
If we decompose $x= z_1 z_2$ with
$z_1=g_1^{m_1}\cdots g_u^{m_{u}}$ and
$z_2=g_{u+1}^{m_{u+1}}\cdots g_r^{m_r}$,
 then
$$
(\nv + \deg(z_1^a))^*\in C_{\betav},
$$
so
$z_1^a \LC{i}{\mathcal M}{M}_{\nv}=0.$
Furthermore
$$
x^a \LC{i}{\mathcal M}{M}_{\nv}=0.
$$

Notice that $a$ does not depends on $\nv$ so
we have proved that $S_{++}\subset rad(Ann_{S}(\LC{i}{\mathcal M}{M}))$, and hence
$$\Gfg(M)\le \gdepth(M).$$

\medskip
\noindent
Now, we prove the other inequality, $\Gfg(M)\ge\gdepth(M)$.

If $S_{++}\subset rad(Ann_{S}(M))$ then there exists $a\in \mathbb N$ such that
for all $x\in S_{++}$, $x^a M=0$.
Since $M$ is finitely generated, by
\lemref{S++^uM=0} there exists a cone $C_{\betav}\subset\mathbb N^r$ with vertex
in some $\betav\in\mathbb N^r$, such that $M_{\nv}=0$ for all $\nv^{*}\in
C_{\betav}$.
Then by \propref{local coho fgrad}, for all $i$ $\LC{i}{\mathcal M}{M}$
is $\Gamma$-finitely graded, so $\Gfg(M)=+\infty\ge \gdepth(M)$.

We can assume that $S_{++}\not\subset rad(Ann_{S}(M))$.
Let
$Ass(M)=\{p_1,\dots,p_t\}$ be the set of  the associated prime ideals of $M$.
Let us consider a minimal
primary decomposition of $0\in M$
$$
0=N_1\cap\dots\cap N_s\cap N_{s+1}\cap\dots\cap N_t
$$
where $Ass(M/N_i)=\{p_i\}$.
We can assume that $p_1,\dots,p_s$ do not contain $S_{++}$, and
$p_{s+1},\dots,p_t$ contain $S_{++}$.

Since the residue field of $S_{\ov}$ is infinite there is an element
 $z\in S_{++}$ such that $z\notin p_1\cup\dots\cup p_s$. We will prove that
$(0:_M z)$ is a $\Gamma$-finitely graded $S$-module.

Since $z\notin p_1\cup\dots\cup p_s$,
then $(0:_M z)\subset N_1\cap\dots\cap N_s$.
In fact, since $N_i$ is a $p_i$-primary submodule of $M$ and $z\notin p_i$,
then $(N_i:_M z)=N_i$.
This last equality is well known: let us assume that there exists
$x\in (N_i:_M z) \setminus N_i$, so $z x \in N_i$.
Since $N_i$ is $p_i$-primary $z^n\in (N_i:_R M)\subset p_i$
for some $n\ge 0$, so $z\in p_i$, contradiction.
Thus, $(0:_M z)\subset(N_i:_M z)=N_i$ for all $i=1,\dots,s$.

On the other hand, by the definition of primary submodule,
$p_i=rad((N_i:_R M))$ for all $i=1,\dots,t$.
 In particular, for $i=s+1,\dots,t$, since
$S_{++}\subset p_i$, there is an $a\in\mathbb N$ such that $S_{++}^aM\subset
N_i$.
Being $M$ finitely generated, by \corref{S++^u(M/N)=0}, there
exists a cone $C_{\betav}\subset\Nr$ with vertex in some $\betav\in\Nr$ such
that $M_{\nv}\subset (N_i)_{\nv}$ for all $\nv^{*}\in C_{\betav}$.

By combining these two facts we get
$$
(0:_M z)_{\nv}\subset
 (N_1 \cap\dots\cap
 N_s \cap N_{s+1}\cap\dots\cap N_t)_{\nv}=0
 $$
 for $\nv^{*}\in C_{\betav}$, so $(0:_M z)$ is
 $\Gamma$-finitely graded. Therefore, $H^i_{\MM}((0:_M z))$ is also $\Gamma$-finitely graded for
 all $i\ge 0$ by \propref{local coho fgrad}.

Since $\Gfg((0:_M z))=+\infty$, from the first part of the proof we get
$\gdepth((0:_M z))=+ \infty$.
Let us consider the exact sequence
$$
0 \longrightarrow (0:_M z)\longrightarrow M \longrightarrow
\frac{M}{(0:_M z)}\longrightarrow 0.
$$
Since $\Gfg((0:_M z))=\gdepth((0:_M z))=+ \infty$ from the long exact sequence
of local cohomology we deduce $\Gfg(M)=\Gfg(M/(0:_M z))$ and
$\gdepth(M)=\gdepth(M/(0:_M z))$. On the other hand there exist $b\in \mathbb
N$ such that $z^b \LC{i}{\mathcal M}{M}=0$ for all $i< \gdepth(M)$. Hence we
may assume that $M$ is a $S$-module for which $z\in S_{++}$ is a non-zero divisor and $z
\LC{i}{\mathcal M}{M}=0$ for all $i< \gdepth(M)$.

We will show by induction on $c$ that if $0\le c\le \gdepth(M)$ then $c\le
\Gamma fg(M)$. The case $c=0$ is trivial. Let us assume that $c>0$, and let us
consider the degree zero exact sequence, $\rv=\deg(z)$,
$$
0 \longrightarrow M(-\rv)\stackrel{. z}{\longrightarrow} M
\longrightarrow \frac{M}{z M}\longrightarrow 0.
$$
From the long exact sequence of local cohomology we get $\gdepth(M)-1\le
\gdepth(M/zM)$, so
$$
0\le c-1\le \gdepth(M)-1\le \gdepth(M/zM).
$$
By induction on
$c$ we get $c-1\le \Gfg(M/zM)$. In particular $\LC{c-2}{\mathcal M}{M/zM}$ is
$\Gamma$-finitely graded. Let us consider the exact sequence on $\nv$, for
$\nv^{*}\in C_{\betav}$,
$$
0=\LC{c-2}{\mathcal M}{M/zM}_{\nv} \longrightarrow
\LC{c-1}{\mathcal M}{M}_{\nv-\rv}
\stackrel{.
z}{\longrightarrow} \LC{c-1}{\mathcal M}{M}_{\nv}.
$$
Since $z \LC{c-1}{\mathcal M}{M}=0$ we deduce
$\LC{c-1}{\mathcal M}{M}$ is $\Gamma$-finitely graded. Hence $c\le \Gfg(M)$.
\end{proof}

\medskip
It is an easy consequence, now, the invariance of $\Gfg$ under Veronese transforms:

\begin{corollary}
\label{fgvero}
Let $S$ be an almost-standard graded ring such that
$S_{\ov}$ is the quotient of a regular ring.
If  $M$ is  a finitely generated $S$-graded module
then
for all $\av, \bv\in\mathbb N ^{*r}$ it holds
$$
\Gfg(M^{(\av,\bv)})=\Gfg(M).
$$
\end{corollary}
\begin{proof}
It follows immediately from \thmref{threetenors} and \corref{fgver}.
\end{proof}


\bigskip

\begin{definition}
Let $M$ be a finitely generated graded $S$-module. We denote by
$$\delta_M:\mathbb N^{*r}\times\mathbb N^{*r} \longrightarrow \mathbb N$$ the numerical
function defined by $\delta_M(\av,\bv)=\depth(M^{(\av,\bv)})$, $\av,\bv\in
\mathbb N^{* r}$. We write $\delta_M(\av)= \delta_M(\av,\ov)$.
\end{definition}

\medskip
Before studying the asymptotic depth of the Veronese of a module, we need a
technical proposition. The following result does not work on the more general
multigraded case, so the restriction to the almost-standard case is necessary.

\begin{proposition}\label{almost-std-fit-to-cone}
Let  $C_{\betav}\subset \Nr$ be a cone  of vertex at $\betav\in\Nr$.
For all $\nv\in\Nr$,
$\bv\in\Zr$ such that $b_i \ge\beta_i$ if $n_i=0$,  and $\av\in\Nr$ such that
$a_i\ge(\beta_i+b_i)/\gamma_i^i$,  $i=1,\dots,r$, we have that
$$(\phi_{\av}(\nv)+\bv)^*\in C_{\betav}.$$

\noindent
In particular, for all $\bv\ge \betav$ and
$\av\in\Nr$ such that
$a_i\ge(\beta_i+b_i)/\gamma_i^i$,  $i=1,\dots,r$, we have that
for all $\nv\in \Zr$
$$
(\phi_{\av}(\nv)+\bv)^*\in C_{\betav}.
$$
\end{proposition}

\begin{proof}
For $\nv\in\Zr$ we have that
$\phi_{\av}(\nv)+\bv=(a_1n_1\gamma_1^1+b_1,\dots,a_rn_r\gamma_r^r+b_r)$ and
hence,
$(\phi_{\av}(\nv)+\bv)^*=(|a_1n_1\gamma_1^1+b_1|,\dots,|a_rn_r\gamma_r^r+b_r|)$.

 We have to find conditions on $\av,\bv\in\mathbb{N}^{*r}$ in order to
assure that $(\phi_{\av}(\nv)+\bv)^*\in C_{\betav}$ for all $\nv\in\Zr$.
So, we
have to impose that for all $i=1,\dots,r$, there exist some
$\lambda_i\in\mathbb R_{\ge 0}$ such that
$|a_in_i\gamma_i^i+b_i|=\beta_i+\lambda_i\gamma_i^i$.
Since
$\gamma_i^i\in\mathbb N^{*}$, then it is only necessary to assure that
$|a_in_i\gamma_i^i+b_i|\ge\beta_i$ for all $i=1,\dots,r$.

If $n_i\neq 0$,
since
$|a_in_i\gamma_i^i+b_i|\ge |a_in_i\gamma_i^i|-|b_i|=|n_i|a_i\gamma_i^i-b_i$, then we
have to impose that
$$|n_i|a_i\gamma_i^i-b_i\ge \beta_i$$
which is equivalent to
$$|n_i|\ge \frac{\beta_i+b_i}{a_i\gamma_i^i}.$$
Hence we must impose that
$$a_i\ge\frac{\beta_i+b_i}{\gamma_i^i}$$
$i=1,\cdots,r$.
If $n_i= 0$ then we have to impose $b_i=|b_i|\ge \beta_i$, $i=1\cdots,r$.

The second part of the result follows from the first one.
\end{proof}

\bigskip

Now, we are ready to prove the theorem that assures constant depth for the
$(\av,\bv)$-Veronese in a region of $\Nr\times\Nr$.

\begin{theorem}
\label{main} Let $M$ be a finitely generated graded $S$-module and let
$s=vad(M^{(*,*)})$. Assume that $S_{\underline{0}}$ is the quotient of a
regular ring. The numerical function $\delta_M$ is asymptotically constant:
there exists
 $\betav\in \Nr$   such that for
all $\bv\ge\betav$ and for all $\av\in\Nr$ such that
$a_i\ge(\beta_i+b_i)/\gamma_i^i$
 it holds
$$
\delta_M(\av,\bv)=s.
$$
\end{theorem}
\begin{proof}
We put  $s=vad(M^{(*,*)})$,  thus
$$
\Gfg(M)=\gdepth(M)=\gdepth(M^{(\av,\bv)})\ge s
$$
by \thmref{threetenors} and \corref{fgver}.
Since $\Gfg(M) \ge  s$ there exist a
 a cone $C_{\betav}\subset\Nr$, $\betav\in\Nr$,  such that $\LC{i}{\mathcal
M}{M}_{\nv}=0$ for all $\nv\in\Zr$ with $\nv^*\in C_{\betav}$ and $i=0,\dots,s-1$.

By
\lemref{almost-std-fit-to-cone}, for $\bv\ge\betav$ and $\av\in\Nr$ such that
$a_i\ge (\beta_i+b_i)/\gamma_i^i$ for all $i=1,\dots,r$, we have that
$(\phi_{\av}(\nv)+\bv)^*\in C_{\betav}$ for all $\nv\in\Zr$.
Hence, we get that for all
$\nv\in\Zr$,
$$
\LC{i}{\mathcal M^{(\av)}}{M^{(\av,\bv)}}_{\nv}=
(\LC{i}{\mathcal M}{M}^{(\av,\bv)})_{\nv}=
(\LC{i}{\mathcal M}{M})_{\phi_{\av}(\nv)+\bv}=0
$$
because $(\phi_{\av}(\nv)+\bv)^*\in C_{\betav}$.
So, we have proved that
$$
\LC{i}{\mathcal M^{(\av)}}{M^{(\av,\bv)}}=0
$$
for $i=0,\dots,s-1$. Therefore,
$$\depth_{\MM^{(\av)}}(M^{(\av,\bv)})\ge s,$$
and by the definition of $s$  we get the claim.
\end{proof}

\bigskip
In the next result we generalize \cite{Eli04}, Proposition 2.1, to general $\mathbb Z$-graded modules.

\begin{proposition}
\label{main2} Let $M$ be a finitely generated graded $S$-module. Let us assume
that $S$ is $\mathbb Z$-graded and that $S_{\underline{0}}$ is the quotient of
a regular ring. The numerical function $\delta_M$ is asymptotically constant:
there exist $s(M)\in \mathbb N$ and $\alpha \in \mathbb N$   such that for all
$a\ge\alpha$ it holds
$$
\delta_M(a)=s(M).
$$
\end{proposition}
\begin{proof}
If  $s=s(M)=vad(M^{(*)})$ then
$$
\Gfg(M)=\gdepth(M)=\gdepth(M^{(a)})\ge s
$$
by \thmref{threetenors} and \corref{fgver}.
Since $\Gfg(M) \ge  s$  there exist an integer
 $\beta\in\mathbb N $,  such that
 $\LC{i}{\mathcal
M}{M}_{n}=0$ for all $n\in\mathbb N$ with $|n|\ge \beta$ and $i=0,\dots,s-1$.
From the first part of \propref{almost-std-fit-to-cone} for all $a\ge \alpha_i= \beta/\gamma_1^1$ we have
that
$$
\LC{i}{\mathcal M^{a}}{M^{(a)}}_{n}=
(\LC{i}{\mathcal M}{M}^{(a)})_{n}=
\LC{i}{\mathcal M}{M}_{a n}=0
$$
for all  $n\neq 0$.
On the other hand we have
$$
\LC{i}{\mathcal M^{a}}{M^{(a)}}_{0}=
(\LC{i}{\mathcal M}{M}^{(a)})_{0}=
\LC{i}{\mathcal M}{M}_{0}=0
$$
for $i=0,\dots,s-1$.
So, we have proved that
$$
\LC{i}{\mathcal M^{(a)}}{M^{(a)}}=0
$$
for $i=0,\dots,s-1$. Therefore,
$$\depth_{\MM^{(a)}}(M^{(a)})\ge s,$$
and by the definition of $s$  we get the claim.
\end{proof}

\begin{corollary}[\cite{Eli04}, Proposition 2.1]
Let $R$ be a Noetherian local ring quotient of a regular ring.
Let $I\subset R$ be an ideal.
Then the depth of $\mathcal R(I)^{(a)}$ is constant for $a\gg 0$.
\end{corollary}

\bigskip

For the multigraded Rees algebra, the best approach to the solution of the
problem is the following proposition.

\begin{proposition}\label{depthmultirees}
If $R$ is the quotient of a regular ring, there exist an integer $s$ and
$\betav\in\Nr$ such that for all $\bv\ge\betav$ and $\av\ge\betav+\bv$ it holds
$$
\depth_{\MM^{(a)}}((I_1^{b_1}\cdots I_r^{b_r}) \mathcal R(I_1^{a_1},\cdots,I_r^{a_r}))=s.
$$
\end{proposition}
\begin{proof}
Notice that, since the Rees algebra $\mathcal R(I_1,\cdots,I_r)$ is standard
multigraded,
$$
\mathcal R(I_1,\cdots,I_r)^{(\av,\bv)}=\mathcal
(I_1^{b_1}\cdots I_r^{b_r}) R(I_1^{a_1},\cdots,I_r^{a_r}),
$$
with $\av=(a_1,\dots,a_r)$ and $\bv=(b_1,\dots,b_r)$.
Now, from \thmref{main} we get the claim.
\end{proof}

See \cite{hyr99} and its reference list for more results on the
Cohen-Macaulay and Gorenstein property of the multigraded Rees
algebras.

\medskip
\baselineskip 8pt

\providecommand{\bysame}{\leavevmode\hbox to3em{\hrulefill}\thinspace}

\end{document}